\documentclass{article}
\usepackage{amsmath}
\usepackage{amssymb,amsfonts}
\usepackage{amsthm}
\usepackage{cite}

\newtheorem{theorem}{Theorem}
\newtheorem{remark}{Remark}

\begin{document}

\title{\bf Lie group classification and invariant exact solutions of the generalized Kompaneets equations}
\author{\bf Oleksii Patsiuk}
\date{}
\maketitle

\begin{center}
Institute of Mathematics, National Academy of Sciences of Ukraine,\\
3 Tereshchenkivs'ka Str., 01601 Kyiv-4, Ukraine\\
\medskip
E-mail: patsyuck@yahoo.com
\end{center}

\begin{abstract}
In this paper, from the group-theoretic point of view it is investigated such class of the generalized Kompaneets equations (GKEs):
$$u_t=\frac1{x^2}\cdot\left[x^4(u_x+f(u))\right]_x, \ (t,x) \in \mathbb{R}_{+} \times \mathbb{R}_{+},$$
where $u=u(t,x)$, $u_t=\frac{\partial u}{\partial t}$, $u_x=\frac{\partial u}{\partial x}$, $u_{xx}=\frac{\partial^2 u}{\partial x^2}$; $f(u)$ is an arbitrary smooth function of the variable $u$. Using the Lie--Ovsiannikov algorithm, the group classification of the class under study is carried out. It is shown that the kernel algebra of the full groups of the GKEs is the one-dimensional Lie algebra $\mathfrak{g}^\cap=\langle \partial_t \rangle$. Using the direct method, the equivalence group $G^\sim$ of the class is found. It is obtained six non-equivalent (up to the equivalence transformations from the group $G^\sim$) GKEs that allow wider invariance algebras than $\mathfrak{g}^\cap$. It is shown that, among the non-linear equations from the class, the GKE with the function $f(u)=u^{\frac43}$ has the maximal symmetry properties, namely, it admits a three-dimensional maximal Lie invariance algebra. Using the obtained operators, it is found all possible non-equivalent group-invariant exact solutions of the GKE under consider.
\end{abstract}

\section{Introduction}\label{1}

In this paper we study a class of the generalized Kompaneets equations (GKEs):
\begin{equation}\label{komp}
u_t=\frac1{x^2}\cdot\left[x^4(u_x+f(u))\right]_x, \ \ (t,x) \in \mathbb{R}_{+} \times \mathbb{R}_{+},
\end{equation}
where $u=u(t,x)$, $u_t=\frac{\partial u}{\partial t}$, $u_x=\frac{\partial u}{\partial x}$, $u_{xx}=\frac{\partial^2 u}{\partial x^2}$; $f(u)$ is an arbitrary smooth function of the variable $u$.

Equation (\ref{komp}) with $f(u)=u^2+u$, namely,
\begin{equation}\label{komp1}
u_t=\frac1{x^2}\cdot\left[x^4(u_x+u^2+u)\right]_x, \ \ (t,x) \in \mathbb{R}_{+} \times \mathbb{R}_{+},
\end{equation}
was obtained in 1950 by A.~S.~Kompaneets \cite{K} (see also \cite{Z}). It describes steading of thermal equilibrium between quanta and electrons in rare plasma (considering only scattering). Its possible applications (mainly, astrophysical) were investigated in detail in \cite{ZS, IS, FR, Z} et al.

If $u\gg1$ then one can put $f(u)=u^2$ (now induced scattering is only considered). The corresponding equation of the form (\ref{komp1}) was studied, e.g., in \cite{Z}. If $u\ll1$ then we get the linear equation with $f(u)=u$. The general solution of this equation was obtained by A.~S.~Kompaneets \cite{K} using the Green function, whose properties have been investigated in \cite{NL, NLG}. The Green function for the linear Kompaneets equation with $f(u)=0$ was obtained in \cite{ZS}.

For the nonlinear Kompaneets equation (\ref{komp1}), not applicable are such classical methods for solving linear partial differential equations (PDEs) as the method of Green function, the method of separation of variables, the method of integral transforms, etc. Therefore, the construction of exact analytical solutions of the nonlinear Kompaneets equation (\ref{komp1}) is an actual task of modern mathematical physics.

One of the most powerful methods for constructing exact solutions of nonlinear PDEs is the classical Lie method \cite{O, O2, BK}, and its various generalizations and modifications (see, e.g., \cite{OV}).

Group analysis of the Kompaneets equation (\ref{komp1}) was held recently in \cite{I}. It was shown that the maximal algebra of invariance (MAI) of this equation is the one-dimensional algebra $\langle\partial_t\rangle$, i.e. equation (\ref{komp1}) only allows the one-parameter time translation group. This symmetry leads only to the well-known stationary solution found by A.~S.~Kompaneets \cite{K}.

At the same time, it has been shown \cite{I} that for various limiting cases (e.g., for the prevailing induced scattering or the degenerate limiting case $u^2 \gg u, \, u^2 \gg u_x$) the corresponding equations of the form (\ref{komp}) allow extensions of the symmetry properties. This allow us to construct a series of new exact solutions, which were not known before.

We note also the recent paper \cite{BTY}, where the analysis of the nonlinear Kompaneets equation (\ref{komp1}) in the case of prevailing induced scattering was held by using the Bluman--Cole method \cite{BC} (see also \cite{OV}) and a number of new exact solutions of this equation was built.

The results of \cite{I, BTY} indicate that in the class of the GKEs (\ref{komp}) there are equations with nontrivial symmetry properties. This enables us to build exact analytical solutions of these equations using the method of symmetry reduction. So, it is naturally arised the problem of classification of symmetry properties of the differential equations of the form (\ref{komp}), i.e. the problem of group classification of the class of the GKEs (\ref{komp}). It can be formulated as follows: find the kernel $\mathfrak{g}^\cap$ of the MAIs of equations from class (\ref{komp}), i.e. the MAIs of equation (\ref{komp}) with an {\it arbitrary} function $f(u)$, and describe {\it all non-equivalent} equations that admit invariance algebras of dimension, higher than $\mathfrak{g}^\cap$.

Hereafter, we are going to work with the class of equations (\ref{komp}), written in the form
\begin{equation}\label{komp2}
u_t=x^2u_{xx}+x\left[xf'(u)+4\right]u_x+4xf(u), \ \ (t,x) \in \mathbb{R}_{+} \times \mathbb{R}_{+}.
\end{equation}
Taking into account the physical meaning of the function $u$, we assume that $u>0$ in (\ref{komp2}).

The purpose of this paper is to carry out the group classification of the class of the GKEs (\ref{komp2}), and to build the exact invariant solutions of the equations admitting the highest symmetry properties.

The structure of this paper is as follows. In Section \ref{2}, using the direct method, we find the complete group of the equivalence transformations of class (\ref{komp2}), up to which we carried out the group classification of one. In Section \ref{3}, using the Lie method, we get the system of the determining equations for the infinitesimal symmetries of equations from class (\ref{komp2}). Analysis of its classifying part is made in Section \ref{4}. In Section \ref{5}, the symmetry reduction is made and exact invariant solutions are built for equation (\ref{komp2}) with $f(u)=u^{\frac43}$, which is a representative of the equivalence class of nonlinear equations (\ref{komp2}) with the three-dimensional MAI.

\section{Group of the equivalence transformations}\label{2}

Performing the group classification of classes of differential equations, it is important to know the local transformations of variables that alter the functional parameters contained in the studied class of equations, but keep the differential structure of one. Such transformations induce an equivalence relation on the set of the functional parameters. In other words, isomorphic are the symmetry groups of two differential equations, which correspond to two different, but equivalent parameters.

Traditionally, finding a group of equivalence transformations, one use the Lie--Ovsyannikov infinitesimal method (see, for example, \cite{O, ITV, M}). However, this method only allows to find all the {\it continuous} equivalence transformations, while for finding a complete group (pseudogroup) of the equivalence transformations (including both the continuous ones and the {\it discrete} ones) it should be used the {\it direct method} \cite{KS}.

We start the construction of the group of the equivalence transformations of the class of the GKEs (\ref{komp2}) from the previous study of the set of {\it admissible} transformations (other names, allowed or form-preserving transformations) of this class of equations. In other words, we look for all non-degenerate point transformations of variables
$$\bar{t} = T(t,x,u), \ \bar{x} = X(t,x,u), \ \bar{u} = U(t,x,u), \ \frac{\partial(T,X,U)}{\partial(t,x,u)} \neq 0,$$
that map a fixed equation of the form (\ref{komp2}) to an equation of the same form:
\begin{equation}\label{komp3}
\bar{u}_{\bar{t}}=\bar{x}^2\bar{u}_{\bar{x}\bar{x}}+
\bar{x}\left[\bar{x}\bar{f}'(\bar{u})+4\right]\bar{u}_{\bar{x}}+4\bar{x}\bar{f}(\bar{u}).
\end{equation}

Without loss of generality, we can restrict ourselves by consideration of point transformations of the form
$$\bar{t}=T(t),\ \bar{x}=X(t,x),\ \bar{u}=U(t,x,u),$$
where $T$, $X$, and $U$ are arbitrary smooth functions of their variables with $T_tX_xU_u\neq0$ (see \cite{KS, PI}). Under these transformations, the partial derivatives are transformed as follows:
\begin{gather*}
u_t=\frac1{U_u}(T_t\bar{u}_{\bar{t}}+X_t\bar{u}_{\bar{x}}-U_t),\quad
u_x=\frac1{U_u}(X_x\bar{u}_{\bar{x}}-U_x),\\
u_{xx}=\frac1{U_u}\left[X_x^2\bar{u}_{\bar{x}\bar{x}}+
\left(X_{xx}-2X_x\frac{U_{xu}}{U_u}+2X_x\frac{U_xU_{uu}}{U_u^2}\right)\bar{u}_{\bar{x}}\right.-{}\\
{}-\left.X_x^2\frac{U_{uu}}{U_u^2}\bar{u}_{\bar{x}}^2-U_{xx}+2\frac{U_xU_{xu}}{U_u}-\frac{U_x^2U_{uu}}{U_u^2}\right].
\end{gather*}

Substituting the last formulas in (\ref{komp2}) and taking into account equality (\ref{komp3}), we obtain the equation:
\begin{gather*}
\bar{u}_{\bar{x}\bar{x}}(x^2X_x^2-T_tX^2)-\bar{u}_{\bar{x}}^2x^2X_x^2\frac{U_{uu}}{U_u^2}+
\bar{u}_{\bar{x}}\left[x^2\left(X_{xx}-2X_x\frac{U_{xu}}{U_u}+2X_x\frac{U_xU_{uu}}{U_u^2}\right)\right.+{}\\
{}+\left.x(xf_u+4)X_x-\left(\bar{f}_{\bar{u}}X+4\right)T_tX-X_t
\vphantom{\frac{U_{xu}}{U_u}}\right]-x^2\left(U_{xx}-2\frac{U_xU_{xu}}{U_u}+\frac{U_x^2U_{uu}}{U_u^2}\right)-{}\\
{}-x(xf_u+4)U_x+4xfU_u-4\bar{f}T_tX+U_t=0.
\end{gather*}
Splitting it in $\bar{u}_{\bar{x}}$, and $\bar{u}_{\bar{x}\bar{x}}$, we have:
\begin{equation}\label{equiv}
\begin{aligned}
\bar{u}_{\bar{x}\bar{x}}:&\ x^2X_x^2-T_tX^2=0,\\
\bar{u}_{\bar{x}}^2:&\ U_{uu}=0,\\
\bar{u}_{\bar{x}}:&\ x^2\left(X_{xx}-2X_x\frac{U_{xu}}{U_u}\right)+x(xf_u+4)X_x-\left(\bar{f}_{\bar{u}}X+4\right)T_tX-X_t=0,\\
1:&\ x^2\left(U_{xx} - 2 U_x \frac{U_{xu}}{U_u}\right)+x(xf_u+4)U_x-4xfU_u+4\bar{f}T_tX-U_t=0\\
\end{aligned}
\end{equation}
(equality $U_{uu}=0$ have been taken into account in the last two equations immediately).

From the second equation of system (\ref{equiv}) we get that
$$U=C_1(t,x)u+C_2(t,x),$$
and from the first one we obtain:
\begin{equation}\label{T}
T_t=\frac{x^2X_x^2}{X^2}.
\end{equation}

Substituting (\ref{T}) in the last equation of system (\ref{equiv}), we have:
\begin{equation}\label{f}
\bar{f}=\frac X{4x^2X_x^2}\left[x^2\left(2\frac{U_xU_{xu}}{U_u}-U_{xx}\right)-x(xf_u+4)U_x+4xfU_u+U_t\right].
\end{equation}
Differentiated equation (\ref{f}) by $u$, we find that
\begin{equation}\label{deriv_f}
\bar{f}_{\bar{u}}=\frac X{4x^2X_x^2U_u}\left[x^2\left(2\frac{U_{xu}^2}{U_u}-U_{xxu}-f_{uu}U_x-f_u U_{xu}\right )-4x \left(U_{xu}-f_uU_u\right)+U_{tu}\right].
\end{equation}

Now we substitute (\ref{T}) and (\ref{deriv_f}) into the third equation of system (\ref{equiv}). After simple transformations, we arrive at the equation:
\begin{gather*}
x^2\left(X_{xx}-2X_x\frac{U_{xu}}{U_u}-\frac{4X_x^2}X\right)-\frac X{4U_u}\left(x^2\left(2\frac{U_{xu}^2}{U_u}-U_{xxu}\right)-4xU_{xu}+U_{tu}\right)+{}\\
{}+4xX_x-X_t+x^2\left(X_x-\frac Xx+\frac{XU_{xu}}{4U_u}\right)f_u+\frac{x^2XU_x}{4U_u}f_{uu}=0.
\end{gather*}

Further analysis of the obtained equation is not possible without additional assumptions on the function $ f(u)$. Thereby, now we find the group of {\it equivalence} transformations of the class of equations (\ref{komp2}).

The equivalence transformations of the class of equations (\ref{komp2}) are picking out the set of all admissible transformations by the additional condition that they map {\it every} equation of the form (\ref{komp2}) to an equation of the same form. In this case, the functional parameter $f$ varies, and thus, the last equation can be split by the derivatives of $f$.

Solving the obtained system of equations we have:
$$X=B(t)x,\quad U=\frac C{B^4(t)}\,u+C_2(t),\quad C\in\mathbb{R}.$$
Then from the first equation of system (\ref{equiv}), it follows that $T=t+A$, where $A\in\mathbb{R}.$

If $B(t)\neq{\rm const}$, then substituting the expressions for $T,X$, and $U$ in (\ref{f}), we get:
$$\bar{f}=\frac C{B^5(t)}f-\frac C{xB^6(t)}u+\frac{C_2'(t)}{4xB(t)}.$$
Multiplying this equality by $4xB(t)$ and differentiating the resulting equation by $x$, we obtain the equality
$$\bar{f}=\frac C{B^5(t)}f,$$
which implies that $C_2'(t)=0$, and $C=0$. But then $U_u=0$, which is impossible.

Thus, $B(t)={\rm const}$. Then from (\ref{f}), we get:
$$\bar{f}=\frac C{B^5}f+\frac{C_2'(t)}{4xB}.$$
Here, as above, we can show that $C_2'(t)=0$. Denoting now $\frac C{B^4}$ in $C_1$, we arrive at the following assertion.

\begin{theorem}\label{T1}
The group of the equivalence transformations $G^\sim$ of the class of the GKEs (\ref{komp2}) consists of the following transformations:
\begin{equation}\label{group}
\bar{t}=t+A,\quad\bar{x}=Bx,\quad\bar{u}=C_1u+C_2,\quad \bar{f}=\frac{C_1}Bf,
\end{equation}
where $A,B,C_1,C_2$ are arbitrary real constants with $BC_1\neq0$.
\end{theorem}

\begin{remark}
From Theorem \ref{T1}, it is directly followed that any transformation $\mathcal{T}$ from the group $G^\sim$ of the equivalence transformations of the class of differential equations (\ref{komp2}) can be represented as the composition
$$\mathcal{T} = \mathcal{T}(A)\mathcal{T}(B)\mathcal{T}(C_1)\mathcal{T}(C_2),$$
where each of the transformations
$$\begin{aligned}
\mathcal{T}(A):&\ (t,x,u,f) \mapsto (t+A,x,u,f),\\
\mathcal{T}(B):&\ (t,x,u,f) \mapsto (t,Bx,u,B^{-1}f),\ B\neq0,\\
\mathcal{T}(C_1):&\ (t,x,u,f) \mapsto (t,x,C_1u,C_1f),\ C_1\neq0,\\
\mathcal{T}(C_2):&\ (t,x,u,f) \mapsto (t,x,u+C_2,f)\\
\end{aligned}$$
belongs to the one-parameter family of equivalence transformations.
\end{remark}

\section{The kernel of MAIs}\label{3}

According to the classical Lie algorithm \cite{O,O2}, we look for the infinitesimal operators generating the invariance algebra of equation (\ref{komp2}) in the class of differential operators of the first order
\begin{equation}\label{2.1}
X = \tau(t,x,u) \partial_t + \xi(t,x,u) \partial_x + \eta(t,x,u) \partial_u,
\end{equation}
where $\tau$, $\xi$, and $\eta$ are arbitrary smooth functions of their variables.

The condition of invariance of equation (\ref{komp2}) with respect to operator (\ref{2.1}) is as follows:
$$\left.X^{(2)}\left\{u_t-x^2u_{xx}-x(xf_u+4)u_x-4xf\right\}\right\vert_{(3)}=0,$$
or, details,
\begin{gather}
\eta^t - 2xu_{xx}\xi -x^2 \eta^{xx} - 2(xf_u+2)u_x\xi-{}\notag\\
{}-\left.x^2f_{uu}u_x\eta-x(xf_u+4)\eta^x-4f\xi-4xf_u\eta\,\right\vert_{(3)}=0.\label{2.2}
\end{gather}
It was used the following notation:
$$X^{(2)}=X+\eta^t\partial_{u_t}+\eta^x\partial_{u_x}+\eta^{tt}\partial_{u_{tt}}+\eta^{tx}\partial_{u_{tx}}+\eta^{xx}\partial_{u_{xx}}$$
is the second prolongation of the operator $X$;
\[\begin{split}
& \eta^i = \mathrm{D}_i(\eta) - u_t \mathrm{D}_i(\tau) - u_x \mathrm{D}_i(\xi),\ i\in\{t,x\},\\
& \eta^{ij} = \mathrm{D}_j(\eta^i) - u_{ti} \mathrm{D}_j(\tau) - u_{ix} \mathrm{D}_j(\xi),\ i,j\in\{t,x\},
\end{split}\]
where $\mathrm{D}_i$, $i\in\{t,x\}$, is the operator of the total differentiation with respect to $i$; condition $\vert_{(3)}$ in (\ref{2.2}) means replacing $u_t$ to $x^2u_{xx}+x(xf_u+4)u_x+4xf$.

Substituting the expressions for $\eta^t$, $\eta^x$, and $\eta^{xx}$ in equation (\ref{2.2}) and splitting the obtained equality with respect to the various derivatives of $u$, we get the following system of determining equations:
\begin{equation}\label{system}
\begin{aligned}
u_xu_{tx}:&\ \tau_u=0,\\
u_{tx}:&\ \tau_x=0,\\
u_xu_{xx}:&\ \xi_u=0,\\
u_x^2:&\ \eta_{uu}=0,\\
u_{xx}:&\ x(2\xi_x-\tau_t)-2\xi=0,\\
u_x:&\ x(xf_u+4)(\tau_t-\xi_x)+2(xf_u+2)\xi+\xi_t+x^2(f_{uu}\eta+2\eta_{xu}-\xi_{xx})=0,\\
1:&\ 4xf(\tau_t-\eta_u)+4f\xi+4xf_u\eta-\eta_t+x(xf_u+4)\eta_x+x^2\eta_{xx}=0.
\end{aligned}
\end{equation}

The last two equations containing the functional parameter $f$ form the {\it classifying part} of the system. Using the remaining equations, we find that
$$\tau=\tau(t),\quad\xi=x\left(\frac12\tau'(t)\ln x+\gamma(t)\right),\quad\eta=\alpha(t,x)u+\beta(t,x),$$
where $\alpha,\beta,\gamma$, and $\tau$ are smooth functions of its variables.

If $f$ is an arbitrary function, we can furter split system (\ref{system}) with respect to derivatives of $f$ and find the {\it kernel} $\mathfrak{g}^\cap$ of MAIs of the equations of the form (\ref{komp2}) (i.e., those operators that are allowed by arbitrary equation from class (\ref{komp2})). Splitting yields: $\xi=\eta=0$, $\tau_t=0$. From the equations, it directly follows the next statement.

\begin{theorem}\label{T2}
The kernel of MAIs of the GKEs (\ref{komp2}) is the one-dimensional Lie algebra $\mathfrak{g}^\cap=\left<\partial_t\right>$.
\end{theorem}

Extensions of the kernel can exist only in the cases when the equations of the classifying part are satisfied not only for an arbitrary function $f$. The analysis of these equations will be made in the next section.

\section{Extensions of the kernel of MAIs}\label{4}

Rewrite the classifying part of system (\ref{system}) as follows
\begin{equation}\label{system_2}
\left\{
\begin{aligned}
&3\tau_t+\tau_{tt}\ln x+4x\alpha_x+2\gamma_t+x[\tau_t(1+\ln x)+2\gamma]f_u+{}\\
&\ {}+2x\beta f_{uu}+2x\alpha uf_{uu}=0,\\
&4\beta_x+x\beta_{xx}-x^{-1}\beta_t+\left(4\alpha_x+x\alpha_{xx}-x^{-1}\alpha_t\right)u+{}\\
&\ {}+2[\tau_t(2+\ln x)+2(\gamma-\alpha)]f+(4\beta+x\beta_x)f_u+(4\alpha+x\alpha_x)uf_u=0.
\end{aligned}
\right.
\end{equation}

For the analysis of system (\ref{system_2}) we apply the method of structural constants. To do this, we first show that this system is equivalent to the system of two {\it ordinary} differential
equations for the function $f(u)$:
\begin{gather}
a+bf_u+cf_{uu}+duf_{uu}=0,\label{class_1}\\
a^*+b^*u+c^*f+d^*f_u+e^*uf_u=0,\label{class_2}
\end{gather}
with constant coefficients $a,b,c,d,a^*,b^*,c^*,d^*$, and $e^*$.

Indeed, since $f$ depends only on $u$, (\ref{system_2}) satisfies only if all the coefficients in these equations are equal to zero, or proportional (with constant coefficients) of some function $\lambda = \lambda(t,x) \not\equiv 0$:
\[\begin{split}
1) \, & 3\tau_t+\tau_{tt}\ln x+4x\alpha_x+2\gamma_t = a \lambda, \ x[\tau_t(1+\ln x)+2\gamma] = b \lambda, \ 2x\beta = c \lambda, \ 2x\alpha = d \lambda,\\
2) \, & 4\beta_x+x\beta_{xx}-x^{-1}\beta_t = a^* \lambda, \ 4\alpha_x+x\alpha_{xx}-x^{-1}\alpha_t = b^* \lambda,\\
& 2[\tau_t(2+\ln x)+2(\gamma-\alpha)] = c^* \lambda, \ 4\beta+x\beta_x = d^* \lambda, \ 4\alpha+x\alpha_x = e^* \lambda.
\end{split}\]
It is easy to verify that if all the coefficients in (\ref{class_1}) and (\ref{class_2}) are simultaneously equal to zero, it corresponds to an arbitrary function $f$. Therefore, extensions of the kernel of MAIs are only possible for the function $f$, which satisfy the overdetermined system of classifying equations of the form (\ref{class_1}) and (\ref{class_2}) with constant coefficients.

The analysis of this system allows to obtain the following result.

\begin{theorem}\label{T3}
The GKE of the form (\ref{komp2}) may allow the invariance algebra of dimension, higher than $\mathfrak{g}^\cap$, when the function $f$ belongs to one of the following classes (non-equivalent up to the transformations from the group $G^\sim$):

1) $f(u)=e^u+ku$ ($k\neq0$);

2) $f(u)=e^u+n$;

3) $f(u)=u\ln u+ku+n$;

4) $f(u)=\ln u+ku+n$;

5) $f(u)=u^m+ku+n$ ($m\neq0,1,2$);

6) $f(u)=u^2+n$;

7) $f(u)=u$;

8) $f(u)=1$;

9) $f(u)=0$,

\noindent where $k,m,n$ are arbitrary real constants.
\end{theorem}

\begin{proof}

Let us consider equation (\ref{class_1}). The analysis of it for the purpose of constructing the general solution strongly depends on the constant $d$.

If $d=0$,  the corresponding equation has the general solutions of the following form:

1) $f(u)=ku^2+lu+n$ if $b=0$;

2) $f(u)=ke^{mu}+lu+n$ (where $km\neq0$) if $b\neq0$,

\noindent where $k,l,m,n$ are arbitrary real constants that satisfy the specified conditions.

Now let $d\neq0$. For the purpose of analysis of equation (\ref{class_1}), we use the fact that the equivalence relations in the class (\ref{komp2}) are transferred to the system of classifying equations (\ref{system_2}), and hence to equation (\ref{class_1}) and (\ref{class_2}). Applying the transformations from the group $G^\sim$ to equation (\ref{class_1}), we find that in the new variables, the structure of this equation is preserved, but its coefficients change as follows:
$$a\mapsto\frac aB,\quad b\mapsto b,\quad c\mapsto cC_1-dC_2,\quad d\mapsto d.$$
Now it is easy to see that picking in the correct way the value of $C_1$, the coefficient $c$ in equation (\ref{class_1}) can be reduced to zero.

Solving the resulting equation (with $c=0$), we obtain the following expressions for its general solution (depending on the coefficients $b$ and $d$):

1) $f(u)=ku\ln u+lu+n$ if $b=0$;

2) $f(u)=k\ln u+lu+n$ (where $k\neq0$) if $b=d$;

3) $f(u)=ku^m+lu+n$ (where $k\neq0$, $m\neq0,1$) in other cases,

\noindent where $k,m,n$ are arbitrary real constants that satisfy the specified conditions.

Now, collecting together all the possible cases for the function $f(u)$ in such a way that they are not mutually disjoint, and taking into account the non-used transformations from the group $G^\sim$, we get the nine non-equivalent (up to the transformations from the group $G^\sim$) classes listed in the formulation of the theorem. If a fixed function $f(u)$ belongs to some of these classes, then the extension of the kernel of MAIs of the corresponding GKE of the form (\ref{komp2}) may be exist.

Analysis of equation (\ref{class_2}) can be made similarly and gives the same result as in the case of equation (\ref{class_1}).
\end{proof}

Substituting the resulting expressions for the function $f(u)$ in system (\ref{system_2}) and holding the corresponding calculations, we arrive at such statement.

\begin{theorem}\label{T4}
All possible MAIs of the GKEs (\ref{komp2}) with some fixed function $f(u)$ are described in Table 1. Any other equation of the form (\ref{komp2}) with nontrivial Lie symmetry maps to one of the equations given in Table 1 by means of the equivalence transformations of the form (\ref{group}).
\end{theorem}

\begin{table}
\caption{The group classification of the GKEs (\ref{komp2})}
\label{tab:1}
\begin{tabular}{lll}
\hline\noalign{\smallskip}
N & $f(u)$ & Basis of $A^{\mathrm{max}}$ \\
\noalign{\smallskip}\hline\noalign{\smallskip}
1 & $e^u$ &  $\partial_t, \, x \partial_x -\partial_u$ \\
2 & $u^k \, \left(k \neq 0,1,\frac43\right)$ &  $\partial_t, \, x \partial_x - \frac{1}{k-1} u \partial_u$ \\
3 & $u^{\frac43}$ & $\partial_t, \, x \partial_x - 3 u \partial_u, \, 2 t \partial_t + (3t+\ln x)x\partial_x - 3 (1+3t+\ln x)u \partial_u$ \\
4 & $u$ & $\partial_t, \, u \partial_u, \varphi(t,x) \partial_u$ \\
5 & $1$ & $\partial_t, \, x\partial_x+u\partial_u,\, (x+u)\partial_u,\, 2t\partial_t+(\ln x-3t)x\partial_x - (\ln x-3t)x\partial_u$, \\
   &     & $4t^2\partial_t+4tx\ln x\partial_x-\left[((\ln x+3t)^2+2t)(x+u)+4tx\ln x\right]\partial_u,$ \\
   &     & $2tx\partial_x-[(\ln x+3t)(x+u)+2tx]\partial_u, \, \psi(t,x)\partial_u$ \\
6 & $0$ & $\partial_t, \, x\partial_x,\, u\partial_u,\, 2t\partial_t+(\ln x-3t)x\partial_x, \, 2tx\partial_x-(\ln x+3t)u\partial_u,$ \\
   &     & $4t^2\partial_t+4tx\ln x\partial_x-\left[(\ln x+3t)^2+2t\right]u\partial_u, \, \psi(t,x)\partial_u$ \\
\noalign{\smallskip}\hline
\end{tabular}
\end{table}

\begin{remark}
The functions $\varphi(t,x)$ and $\psi(t,x)$ in Table \ref{tab:1} are arbitrary smooth solutions of the equations
$$u_t=x^2u_{xx}+x(x+4)u_x+4xu,$$
and
$$u_t=x^2u_{xx}+4xu_x,$$
respectively.
\end{remark}

\begin{proof}

{\bf I.} Substituting the function $f(u)=e^u+ku$ ($k\neq0$) in the first equation of system (\ref{system_2}) and splitting it in $e^u$ and $ue^u$, we obtain:
\begin{equation}\label{split_11}
\left\{
\begin{aligned}
&\alpha=0,\\
&\tau_t(1+\ln x)+2(\gamma+\beta)=0,\\
&3\tau_t+\tau_{tt}\ln x+2\gamma_t+kx[\tau_t(1+\ln x)+2\gamma]=0.
\end{aligned}
\right.
\end{equation}

Further, substituting the function $f(u)$ in the second equation of system (\ref{system_2}) (and taking into account the equality $\alpha=0$) and splitting it in $u$ and $e^u$, we have:
\begin{equation}\label{split_12}
\left\{
\begin{aligned}
&\tau_t(2+\ln x)+2\gamma=0,\\
&4\beta+x\beta_x=0,\\
&4\beta_x+x\beta_{xx}-x^{-1}\beta_t=0.
\end{aligned}
\right.
\end{equation}

Solving the system of equations (\ref{split_11}) and (\ref{split_12}), we get: $\tau=C_1=\textrm{const}$, $\alpha = \beta = \gamma=0$. Then $\xi=0$, and $\eta=0$. Thus, for the function $f(u)=e^u+ku \, (k\neq0)$, MAI of the corresponding GKE (\ref{komp2}) is the one-dimensional Lie algebra $\mathfrak{g}^\cap$.

{\bf II.} Substituting the function $f(u)=e^u+n$ in the first equation of system (\ref{system_2}) and splitting it in $e^u$ and $ue^u$, we obtain:
\begin{equation}\label{split_21}
\left\{
\begin{aligned}
&\alpha=0,\\
&\tau_t(1+\ln x)+2(\gamma+\beta)=0,\\
&3\tau_t+\tau_{tt}\ln x+2\gamma_t=0.
\end{aligned}
\right.
\end{equation}

Further, substituting the function $f(u)$ in the second equation of system (\ref{system_2}) (and taking into account the equality $\alpha=0$) and splitting it in $u$ and $e^u$, we have:
\begin{equation}\label{split_22}
\left\{
\begin{aligned}
&2[\tau_t(2+\ln x)+2\gamma]+4\beta+x\beta_x=0,\\
&4\beta_x+x\beta_{xx}-x^{-1}\beta_t+2n[\tau_t(2+\ln x)+2\gamma]=0.
\end{aligned}
\right.
\end{equation}

Analysis of the system of equations (\ref{split_21}) and (\ref{split_22}) strongly depends on the constant $n$, namely, the cases $n\neq0$ and $n=0$ yield different results. When $n\neq0$ we obtain: $\tau=C_1=\textrm{const}$, $\alpha = \beta = \gamma=0$, so in this case, there are no extensions of the kernel $\mathfrak{g}^\cap$. When $n=0$ we get: $\tau=C_1=\textrm{const}$, $\alpha = 0$, $\beta=C_2=\textrm{const}$, $\gamma=-C_2$. Hence
$$\tau=C_1, \ \xi=C_2 x, \ \eta=-C_2.$$
Thus, for the function $f(u)=e^u$, MAI of the corresponding GKE (\ref{komp2}) is the two-dimensional Lie algebra
$\langle \partial_t, \, x \partial_x - \partial_u \rangle$ (case 1 of Table \ref{tab:1}).

{\bf III--IV.} Substituting the functions $f(u)=u\ln u+ku+n$ and $f(u)=\ln u+ku+n$ to system (\ref{system_2}) and holding appropriate splittings and calculations, we obtain that $\tau=C_1=\textrm{const},\ \xi=\eta=0$. Consequently, MAI of the corresponding GKE (\ref{komp2}) is the one-dimensional Lie algebra $\mathfrak{g}^\cap$.

{\bf V.} Substituting the function $f(u)=u^m+ku+n$ ($m\neq0,1,2$) in the first equation of system (\ref{system_2}) and splitting it by the powers of $u$, we obtain:
\begin{equation}\label{split_4}
\left\{
\begin{aligned}
&\beta=0,\\
&\tau_t(1+\ln x)+2\gamma+2(m-1)\alpha=0,\\
&3\tau_t+\tau_{tt}\ln x+4x\alpha_x+2\gamma_t+kx[\tau_t(1+\ln x)+2\gamma]=0.
\end{aligned}
\right.
\end{equation}

Now we substitute $f(u)$ in the second equation of system (\ref{system_2}) (taking into account the first two equality of (\ref{split_4})) and also split it by the powers of $u$. We have:
\begin{equation}\label{split_4a}
\left\{
\begin{aligned}
&n(\tau_t-2m\alpha)=0,\\
&2\tau_t+mx\alpha_x=0,\\
&2k\tau_t-4k(m-1)\alpha-x^{-1}\alpha_t+(4+kx)\alpha_x+x\alpha_{xx}=0.
\end{aligned}
\right.
\end{equation}

{\bf V.A.} If $n\neq0$, then the system of equations (\ref{split_4}) and (\ref{split_4a}) implies that $\tau=C_1=\mathrm{const}$, $\xi=\eta=0$. So, in this case there are no extensions of the kernel $\mathfrak{g}^\cap$.

{\bf V.B.} If $n=0$, then solving system (\ref{split_4a}), we find that
\begin{equation}\label{taualpha}
\tau_t=-\frac12mC_1,\quad\alpha=C_1(\ln x+3t)+C_2,
\end{equation}
where $C_i=\mathrm{const}$ ($i=1,2$), obeying the conditions $kC_1=0$ and $kC_2=0$. It follows that the system (\ref{split_4}) can allow different solutions depending on the constant $k$.

First, let us consider the case $k=0$. Substituting (\ref{taualpha}) in the second and third equations of system (\ref{split_4}) (and splitting the second equation by $\ln x$), we get:
\begin{equation}\label{k_0}
\left\{
\begin{aligned}
&(3m-4)c_1=0,\\
&mc_1-4(m-1)c_2-12(m-1)c_1t-4\gamma=0,\\
&(3m-8)c_1-4\gamma_t=0.
\end{aligned}
\right.
\end{equation}

If $m\neq\frac43$, then $C_1=0,\ \gamma=-(m-1)C_2$. Hence
$$\tau=C_0=\textrm{const}, \ \xi=-(m-1)C_2x, \ \eta=C_2u.$$
So, if $f(u)=u^m$ ($m\neq0,1,2,\frac43$), then MAI of the corresponding GKE (\ref{komp2}) is the two-dimensional Lie algebra $\langle \partial_t, \, (m-1)x\partial_x-u\partial_u \rangle$ (this gives us the case 2 of Table \ref{tab:1}; below, we will show that it also includes the case $f(u)=u^2$).

If $m=\frac43$, then from system (\ref{k_0}) we obtain that $\gamma=\frac13(C_1-C_2)-C_1t$. Hence
$$\tau=C_0-\frac23C_1t, \ \xi=\frac13x(C_1-C_2-3C_1t-C_1\ln x), \ \eta=(3C_1t+C_1\ln x+C_2)u.$$
Thus, it was proved that if $f(u)=u^{\frac43}$ then MAI of the corresponding GKE (\ref{komp2}) is the three-dimensional Lie algebra $\langle \partial_t, \, 2t\partial_t+(3t+\ln x-1)x\partial_x-3(3t+\ln x)u\partial_u, \, x\partial_x-3u\partial_u \rangle$ (this gives us the case 3 of Table \ref{tab:1}).

The consideration of the case $k\neq0$ leads to the equalities $\tau=C_1=\mathrm{const}$, $\xi=\eta=0$. So, in this case there are no extensions of the kernel $\mathfrak{g}^\cap$.

{\bf VI.} Substituting the function $f(u)=u^2+n$ in the first equation of system (\ref{system_2}) and splitting it by the powers of  $u$, we obtain:
\begin{equation}\label{split_4b}
\left\{
\begin{aligned}
&\tau_t(1+\ln x)+2(\gamma+\alpha)=0,\\
&3\tau_t+\tau_{tt}\ln x+4x\alpha_x+2\gamma_t+4x\beta=0.
\end{aligned}
\right.
\end{equation}

Now we express the function $\alpha = \alpha(t,x)$ from the first equation of system (\ref{split_4b}):
$$\alpha = -\frac{\tau_t}{2}(1+\ln x)-\gamma,$$
and find its partial derivatives $\alpha_t,\alpha_x,\alpha_{xx}$:
$$\alpha_t = -\frac{\tau_{tt}}{2}(1+\ln x)-\gamma_t, \ \alpha_x = -\frac{\tau_t}{2x}, \ \alpha_{xx} = \frac{\tau_t}{2x^2}.$$
Substituting these expressions and the function $f(u) $ in the second equation of system (\ref{system_2}) and splitting it by the powers of $u$, we have:
\begin{equation}\label{split_4c}
\left\{
\begin{aligned}
&\tau_t=0,\\
&8\beta+2x\beta_x+x^{-1}\gamma_t=0,\\
&4\beta_x+x\beta_{xx}-x^{-1}\beta_t+8n\gamma=0.
\end{aligned}
\right.
\end{equation}

In view of $\tau_t=0$, from system (\ref{split_4b}) we get:
$$\alpha=-\gamma,\quad\beta=-\frac1{2x}\gamma_t.$$
Substituting the expression for $\beta$ to the last two equations of system (\ref{split_4c}), we obtain:
$$\gamma_t=0,\quad n\gamma=0.$$

If $n\neq0$, then $\gamma=0$, and this implies that $\xi=0$, and $\eta=0$. In view of $\tau_t=0$, we conclude that in this case there are no extensions of the kernel $\mathfrak{g}^\cap$.

If $n=0$, then $\gamma=C_1=\mathrm{const}$, and hence
$$\tau = C_0, \ \xi=C_1x, \ \eta=-C_1u.$$
Consequently, for the function $f(u)=u^2$, MAI of the corresponding GKE (\ref{komp2}) is the two-dimensional Lie algebra $\langle \partial_t, \, x\partial_x-u\partial_u \rangle$ (obviously, this case may be included in the case 2 of Table \ref{tab:1}).

{\bf VII.} Let $f(u)=u$. Then system (\ref{system_2}) (after splitting the second equation in $u$) takes the form:
\begin{equation}\label{split_7}
\left\{
\begin{aligned}
&4x\alpha_x+3\tau_t+\tau_{tt}\ln x+2\gamma_t+x[\tau_t(1+\ln x)+2\gamma]=0,\\
&(4+x)\alpha_x+x\alpha_{xx}-x^{-1}\alpha_t+2[\tau_t(2+\ln x)+2\gamma]=0,\\
&4\beta+(4+x)\beta_x+x\beta_{xx}-x^{-1}\beta_t=0.
\end{aligned}
\right.
\end{equation}

Expressing $\alpha_x$ from the first equation of the last system:
$$\alpha_x = -\frac{1}{4x}\left\{3 \tau_t + 2 \gamma_t + \tau_{tt} \ln x + x [\tau_t (1+\ln x)+2\gamma] \right\},$$
and integrating the obtained expression by the variable $x$, we have:
$$\alpha=C(t)-\frac14\left[2\gamma x+(3\tau_t+2\gamma_t)\ln x+\tau_t x\ln x+\frac12\tau_{tt}\ln^2x\right],$$
where $C=C(t)$ is an arbitrary function of the variable $t$. Hence we find the partial derivatives $\alpha_t$, and $\alpha_{xx}$:
\begin{gather*}
\alpha_t=C'(t)-\frac14\left[2\gamma_t x+(3\tau_{tt}+2\gamma_{tt})\ln x+\tau_{tt} x\ln x+\frac12\tau_{ttt}\ln^2x\right],\\
\alpha_{xx}=\frac1{4x^2}[3\tau_t+2\gamma_t+\tau_{tt} (\ln x-1) -\tau_tx].
\end{gather*}

Substituting the expressions for $\alpha_t$, $\alpha_x$, $\alpha_{xx}$ in the second equation of system (\ref{split_7}) and splitting it by the functions of $x$, we get:
$$\tau=C_1=\mathrm{const},\quad\gamma=0,\quad\alpha=C_2=\mathrm{const}.$$
Hence we obtain:
$$\tau=C_1=\mathrm{const},\ \xi=0,\ \eta=C_2u+\varphi(t,x),$$
where the function $\varphi=\varphi(t,x)$ is an arbitrary smooth solution of the equation
$$\varphi_t=x^2\varphi_{xx}+x(x+4)\varphi_x+4x\varphi.$$

Consequently, for the function $f(u)=u$, MAI of the corresponding GKE (\ref{komp2}) is the infinite-dimensional Lie algebra, which is the semi-direct sum of the two-dimensional solvable Lie algebra $\langle \partial_t, u\partial_u \rangle$ and the infinite-dimensional trivial ideal $\langle \varphi(t,x)\partial_u \rangle$ (this is the case 4 of Table \ref{tab:1}).

{\bf VIII.} Let $f(u)=0$. Then system (\ref{system_2}) after splitting the second equation in $u$ takes the form:
$$\left\{
\begin{aligned}
&4x\alpha_x+3\tau_t+\tau_{tt}\ln x+2\gamma_t=0,\\
&4\alpha_x+x\alpha_{xx}-x^{-1}\alpha_t=0,\\
&4\beta_x+x\beta_{xx}-x^{-1}\beta_t=0.
\end{aligned}
\right.$$

Analysis of the resulting system of equations is performed similarly to the case {\bf VII.} As a result, we obtain:
\begin{gather*}
\tau=C_1+C_2t+C_3t^2,\quad\gamma=C_4+C_5t,\\
\alpha=C_6-\frac14(9C_2+2C_3+6C_5)t-\frac94C_3t^2-\frac14\left[(3C_2+2C_5+6C_3t)\ln x+C_3\ln^2x\right],
\end{gather*}
where $C_i \, (i=1, \ldots, 6)$ are arbitrary real constants.

Hence we get:
\begin{gather*}
\tau=C_1+C_2t+C_3t^2,\quad\xi=x\left[C_4+\frac12C_2\ln x+(C_5+C_3\ln x)t\right],\quad\eta=\psi(t,x)+{}\\
{}+\left[C_6-\frac14\{(3C_2+2C_5)\ln x+C_3\ln^2x+(9C_2+2C_3+6C_5+6C_3\ln x)t+9C_3t^2\}\right]u,
\end{gather*}
where the function $\psi=\psi(t,x)$ is an arbitrary smooth solution of the equation
\begin{equation}\label{last}
\psi_t=x^2\psi_{xx}+4x\psi_x.
\end{equation}

Consequently, for the function $f(u)=0$, MAI of the corresponding GKE (\ref{komp2}) is the infinite-dimensional Lie algebra, which is the semi-direct sum of the six-dimensional Lie algebra with the basis operators
\begin{gather*}
X_1 = \partial_t, \ X_2 = x \partial_x, \ X_3 = u \partial_u,\ X_4=2tx \partial_x - (\ln x + 3t) u \partial_u,\\
X_5=4t\partial_t+2x\ln x\partial_x-3(\ln x+3t)u\partial_u,\\
X_6 = 4t^2 \partial_t + 4tx \ln x\partial_x - \left[(\ln x + 3t)^2 + 2t\right]u \partial_u,
\end{gather*}
and the infinite-dimensional trivial ideal $\langle \psi(t,x)\partial_u \rangle$ (this gives us the case 6 of Table \ref{tab:1}).

{\bf IX.} If $f(u)=1$, for finding MAI of the corresponding GKE (\ref{komp2}) we use the fact that this equation can be mapped to equation (\ref{komp2}) with $f(u)=0$ using the transformation of variables:
\begin{equation}\label{change}
\overline{t}=t,\ \overline{x}=x,\ \overline{u}=u+x.
\end{equation}
Under these transformation, the differentiation operators are transformed as follows:
$$\partial_{\overline{t}}=\partial_t,\quad\partial_{\overline{x}}=\partial_x-\partial_u,\quad\partial_{\overline{u}}=\partial_u.$$

It follows that for the function $f(u)=1$, MAI of the corresponding GKE (\ref{komp2}) is the infinite-dimensional Lie algebra with the basis operators
\begin{gather*}
X_1 = \partial_t, \ X_2 = x \partial_x - x \partial_u, \ X_3 = (x+u) \partial_u,\ X_4 = 2tx \partial_x - [(\ln x + 3t) (x+u) + 2tx] \partial_u,\\
X_5 = 4t \partial_t + 2x\ln x\partial_x - [3(\ln x+3t)(x+u)+2x\ln x]\partial_u,\\
X_6 = 4t^2 \partial_t + 4tx \ln x\partial_x - \left[((\ln x + 3t)^2 + 2t)(x+u) + 4tx \ln x\right]\partial_u,\\
\textrm{and}\quad X_{\infty} = \psi(t,x)\partial_u,
\end{gather*}
where the function $\psi=\psi(t,x)$ is an arbitrary smooth solution of the equation (\ref{last}) (this gives us the case 5 of Table \ref{tab:1}).
\end{proof}

\begin{remark}
It should be emphasized that Theorem \ref{T4} gives a comprehensive description of all non-equivalent GKEs of the form (\ref{komp2}) {\it up to the transformations of variables from the group $G^\sim$.}  However, prooving the theorem, it was shown that the GKE (\ref{komp2}) with $f(u)=1$ may be mapped to the GKE (\ref{komp2}) with $f(u)=0$ by the local transformation (\ref{change}), which does not belong to the group $G^\sim$.

Note also that the GKEs (\ref{komp2}) with $f(u)=u^k$ ($k \neq 0,1,\frac43$) and $f(u)=e^u$ have the isomorphic MAIs (type $A_ {2.1}$ by Mubarakzyanov's classification {\rm \cite{M2}}). However, the direct analysis of (\ref{f}) shows that among the admissible transformations in the class of the GKEs of the form (\ref{komp2}) there is no such point transformations of variables mapping these two equations into each other.
\end{remark}

\begin{remark}
Lie's classification of the linear parabolic second order differential equations with two independent variables is widely known (see, e.g., {\rm \cite{L}}). One of the main results of this classification is the fact that any equation of the form
$$P(t,x) u_t + Q(t,x) u_x + R(t,x) u_{xx} + S(t,x) u = 0, \ \ P \neq 0, \ R \neq 0,$$
which admits a five-dimensional nontrivial Lie algebra of infinitesimal symmetries is reduced to the linear heat equation
\begin{equation}\label{rem2}
v_{\tau}=v_{yy}
\end{equation}
by the change of variables
\begin{equation}\label{rem1}
\tau=\alpha(t), \ y=\beta(t,x), \ v=\gamma(t,x)u, \ \ \alpha_t \beta_x \neq 0.
\end{equation}

Therefore, Theorem \ref{T4} implies that equation (\ref{komp2}) with $f(u)=\mathrm{const}$ admitting the five-dimensional non-trivial algebra of infinitesimal symmetries is reduced to the linear heat equation by the change of variables (\ref{rem1}). Indeed, it has been shown {\rm \cite{ZS}} that the linear Kompaneets equation (\ref{komp2}) with $f(u)=0$ reduces to equation (\ref{rem2}) by the change of variables
$$\tau=t, \ y=3t+\ln x, \ v=u.$$
\end{remark}

\section{Symmetry reduction and exact solutions}\label{5}

In the previous section, it was shown that the equation
\begin{equation}\label{eq_10}
u_t=x^2u_{xx}+4x\left(\frac13x\,u^{\frac13}+1\right)u_x+4xu^{\frac43}, \ \ (t,x) \in \mathbb{R}_{+} \times \mathbb{R}_{+}
\end{equation}
(corresponding to the case 3 of Table \ref{tab:1}), has the highest symmetry properties among the GKEs of the form (\ref{komp2}), namely, this equation admits as MAI the three-dimensional Lie algebra of infinitesimal symmetries with the basis operators:
\begin{equation}\label{E0}
X_1=\partial_t, \ X_2=x\partial_x-3u\partial_u, \ X_3=t\partial_t+\frac12(3t+\ln x)x \partial_x-\frac32(1+3t+\ln x)u\partial_u.
\end{equation}

We perform an exhaustive analysis of all possible non-equivalent invariant solutions of equation (\ref{eq_10}), which can be constructed by operators of the algebra $\langle X_1, X_2, X_3 \rangle$.

It is easy to show that this algebra is isomorphic to the Lie algebra $A_{3.4}^{\frac12}$ by Mubarakzyanov's classification \cite{M2}. Really, the algebra $A_{3.4}^{\frac12} = \langle e_1, e_2, e_3 \rangle$ is determined by the following commutation relations:
\begin{equation}\label{E1}
[e_1,e_2]=0,\ [e_1,e_3]=e_1,\ [e_2,e_3]=\frac12\,e_2.
\end{equation}
Choosing operators
\begin{equation}\label{E2}
e_1=X_1+3X_2, \ e_2=X_2, \ e_3=X_3
\end{equation}
as the basis operators of the algebra $\langle X_1, X_2, X_3 \rangle$, we can immediately establish that they satisfy the commutation relations (\ref{E1}).

It is known that a comprehensive list of the non-equivalent invariant solutions of some PDE by operators from the admitted Lie algebra can be obtained by constructing an optimal system of subalgebras of the Lie algebra (see, e.g., \cite[Section 3.3]{O2}). Since MAI of equation (\ref{eq_10}) is the three-dimensional, for the construction of its optimal system of subalgebras we may use the results of \cite{PW}, where the subgroup analysis of all real low-dimensional Lie algebras was held.

Having a complete list of all non-equivalent subalgebras of the algebra $A_{3.4}^{\frac12}$ (up to conjugacy, which is determined by actions of the group of inner automorphisms of the algebra), and taking into account (\ref{E2}), we arrive at the following statement.

\begin{theorem}\label{T5}
The optimal system of subalgebras of MAI of equation (\ref{eq_10}) consists of the following ones:

one-dimensional: $\langle X_2 \rangle$, $\langle X_3 \rangle$, $\langle X_1+2X_2 \rangle$, $\langle X_1+3X_2 \rangle$, $\langle X_1+4X_2 \rangle$;

two-dimensional: $\langle X_1+3X_2, X_2 \rangle$, $\langle X_1+3X_2, X_3 \rangle$, $\langle X_2, X_3 \rangle$;

three-dimensional: $\langle X_1, X_2, X_3 \rangle$,

\noindent where the operators $X_1$, $X_2$, $X_3$ are of the form (\ref{E0}).
\end{theorem}

Since equation (\ref{eq_10}) is a PDE with two independent variables $t$ and $x$, then it can be look for the invariant solutions of ranks $\rho=0$ or $\rho=1$.

First, we perform the detailed analysis of invariant solutions of rank $\rho=1$, which is based on the one-dimensional Lie algebras listed in Theorem \ref{T5}. Note that all these algebras satisfy the necessary conditions for existence of the non-degenerate invariant solutions.

{\bf I.} {\it Algebra} $\left<X_2\right>$. The equation $X_2F(t,x,u)=0$ has the solutions $\omega_1=t,\ \omega_2=ux^3$, and therefore, the corresponding ansatz is of the form
$$u=x^{-3}\varphi(t)$$
and reduces equation (\ref{eq_10}) to the equation $\varphi'=0$, hence $\varphi(t)=c=\mathrm{const}$. In this case, we obtain the one-parameter family of the stationary solutions of equation (\ref{eq_10}):
$$u=c \, x^{-3}.$$

{\bf II.} {\it Algebra} $\left<X_3\right>$. In this case, the equation $X_3F(t,x,u)=0$ has the solutions $\omega_1=\frac1{\sqrt{t}}(\ln x-3t)$, and $\omega_2=ux^3\sqrt{t^3}$. Then the corresponding ansatz reads as
$$u=\frac{\varphi(y)}{x^3\sqrt{t^3}},\quad y=\frac1{\sqrt{t}}(\ln x-3t)$$
and reduces equation (\ref{eq_10}) to the equation
$$\varphi''+\left(\frac43\varphi^{\frac13}+\frac12y\right)\varphi'+\frac32\varphi=0.$$

We could not find the general solution of this equation, however, it easy to see that one has a particular solution $\varphi=27y^{-3}$. Then we have the following particular solution of equation (\ref{eq_10}):
$$u(t,x)=\frac{27}{x^3(\ln x-3t)^3}.$$

{\bf III.} {\it Algebra} $\left<X_1+2X_2\right>$. In this case, the equation $(X_1+2X_2)F(t,x,u)=0$ has the solutions $\omega_1=xe^{-2t}$, and $\omega_2=ux^3$. Then the corresponding ansatz reads as
$$u=x^{-3}\varphi(xe^{-2t})$$
and reduces equation (\ref{eq_10}) to the equation
$$\varphi''+\frac4{3y}\varphi^{\frac13}\varphi'=0,\quad\varphi=\varphi(y),\ y=xe^{-2t}.$$

We have got a generalized Emden-Fowler equation, for which the solution in the parametric form is known \cite[Subs.~2.5.2,\ No.~5]{ZP}\footnote{In the solution of this equation there is a mistake. It should be read: $x$ as in \cite{ZP}, and $y$ as $(A^{-1}\tau)^{\frac1m}.$}. We rewrite it as a function $y(\varphi)$:
$$y=\pm\exp\left(c_1-\int\frac{d\varphi}{\sqrt[3]{\varphi^4}-\varphi+c_2}\right).$$

If $c_2=0$, we have: $y=c_3\left(1-\frac1{\sqrt[3]{\varphi}}\right)^{-3}$, hence $\varphi=\frac1{\left(1-c\sqrt[3]{y^{-1}}\right)^3}$. Then we obtain the one-parameter family of the particular solutions of equation (\ref{eq_10}):
$$u(t,x)=\frac1{x^3\left(1-c\sqrt[3]{x^{-1}e^{2t}}\right)^3}.$$

{\bf IV.} {\it Algebra} $\left<X_1+3X_2\right>$. Since the equation $(X_1+3X_2)F(t,x,u)=0$ has two solutions: $\omega_1=xe^{-3t},\ \omega_2=ux^3$, the corresponding ansatz reads as
$$u=x^{-3}\varphi(xe^{-3t})$$
and reduces equation (\ref{eq_10}) to the equation
$$\varphi''+\frac1y\left(\frac43\varphi^{\frac13}+1\right)\varphi'=0,\quad \varphi=\varphi(y),\ y=xe^{-3t}.$$

We could not find the solutions of this equation explicitly, but it is possible to express $y$ as a function of $\varphi$:
\begin{gather*}
y_1=\pm\exp\left[\frac34\,c_1\left(2\arctan\frac1{c_1\sqrt[3]{\varphi}}+
\ln\left|\frac{c_1\sqrt[3]{\varphi}+1}{c_1\sqrt[3]{\varphi}-1}\right|\right)+c_2\right],\\
y_2=\pm\exp\left[\frac34\,c_3\left(\vphantom{\frac{c_3^2\sqrt[3]{\varphi^2}}{c_3^2\sqrt[3]{\varphi^2}}}
2\arctan\left(\frac1{c_3\sqrt[3]{\varphi}}+1\right)+2\arctan\left(\frac1{c_3\sqrt[3]{\varphi}}-1\right)+{}\right.\right.\\
\left.\left.{}+\ln\frac{2c_3^2\sqrt[3]{\varphi^2}+
2c_3\sqrt[3]{\varphi}+1}{2c_3^2\sqrt[3]{\varphi^2}-2c_3\sqrt[3]{\varphi}+1}\right)+c_4\right].
\end{gather*}

{\bf V.} {\it Algebra} $\left<X_1+4X_2\right>$. In this case, the equation $(X_1+4X_2)F(t,x,u)=0$ has the solutions $\omega_1=xe^{-4t}$, and $\omega_2=ux^3$, hence the corresponding ansatz is of the form
$$u=x^{-3}\varphi(xe^{-4t})$$
and reduces equation (\ref{eq_10}) to the equation
$$\varphi''+\frac2y\left(\frac23\varphi^{\frac13}+1\right)\varphi'=0,\quad\varphi=\varphi(y),\ y=xe^{-4t}.$$

For this equation, the solution in the parametric form is known \cite[Subs.~2.6.2,\ No.~98]{ZP}. We rewrite it as a function $y(\varphi)$:
$$y=c_1\exp\left(\int\frac{d\varphi}{c_2-\sqrt[3]{\varphi^4}-\varphi}\right).$$

Putting $c_2=0$, we get: $y=c_1\left(1+\frac1{\sqrt[3]{\varphi}}\right)^3$, hence $\varphi=\left(c\sqrt[3]{y}-1\right)^{-3}$. Then we obtain the one-parameter family of the particular solutions of equation (\ref{eq_10}):
$$u(t,x)=\frac1{x^3\left(c\sqrt[3]{xe^{-4t}}-1\right)^3}.$$

To find the invariant solutions of equation (\ref{eq_10}) of rank $\rho=0$, it is necessary to use the two-dimensional subalgebras of the algebra $\langle X_1, X_2, X_3 \rangle$. By Theorem \ref{T5}, there are three non-equivalent (up to conjugacy, which is determined by the actions of the group of inner automorphisms of the algebra $\langle X_1, X_2, X_3 \rangle$) two-dimensional subalgebras of this algebra. However, their analysis shows that they do not lead to any new invariant solutions of equation (\ref{eq_10}) compared with ones obtained as a result of the analysis of the one-dimensional subalgebras.

\section{Conclusions}\label{6}

In the present paper, the Lie group classification problem for the class of the GKEs (\ref{komp2}) was solved exhaustively. The main result of the paper is the classification list (see Table \ref{tab:1}), which consists of the six non-equivalent cases (up to equivalence transformations obtained in Section \ref{2}). Among the corresponding non-linear equations from the class under study, the GKE with $f(u)=u^{\frac43}$ has the maximal symmetry properties, namely, it admits a three-dimensional MAI. All possible non-equivalent exact invariant solutions for this equation were constructed as illustrative examples. In the forthcoming article, we are going to consider from group-theoretic point of view one class of the variable coefficients GKEs with two functional parameters.

\medspace
{\bf Acknowledgements.} The author wish to thank Sergii Kovalenko for the useful comments.

\end{document}